\documentclass[a4paper, 12pt]{article}

\usepackage[latin1]{inputenc}
\usepackage[dvips]{graphics}
\usepackage{epsfig}
\usepackage{bbm}
\usepackage{wrapfig}
\usepackage{float}
\usepackage{setspace}
\usepackage{picins}
\usepackage{amssymb}
\usepackage{graphics}
\usepackage{mathrsfs}
\usepackage{amsthm}
\usepackage{amsmath}

\normalsize

\newtheorem{theorem}{Theorem}

\newtheorem{lemma}{Lemma}

\theoremstyle{remark}

\title{A note on the abelian sandpile in $\mathbbm{Z}^d$}
\author{Mykhaylo Tyomkyn}

\begin{document}

\maketitle

\begin{abstract}
We analyse the abelian sandpile model on $\mathbbm{Z}^d$ for the starting configuration of $n$ particles in the origin and $2d-2$ particles otherwise. We give a new short proof of the theorem of Fey, Levine and Peres \cite{FLP} that the radius of the toppled cluster of this configuration is $O(n^{1/d})$. 
\end{abstract}

Consider the grid $\mathbbm{Z}^2$ with an integer written inside each cell. We call this a \emph{configuration} and think of the numbers as particles. We want to study a one player game with the following rule: whenever we see at least $4$ particles in a cell $(x,y)$, we are allowed to \emph{topple} this cell, i.e.\ to remove $4$ particles from $(x,y)$ and add one particle to each of its neighbour cells $(x, y-1)$, $(x, y+1)$, $(x-1, y)$, $(x+1, y)$. The order in which we topple \emph{active} cells, that is cells with at least $4$ particles, is subject to our choice. Once no cell contains more than $3$ particles, the procedure stops. Cells with less than $4$ particles are called \emph{stable} and the whole procedure is called \emph{stabilisation}.

What we just described is called \emph{the abelian sandpile model}. It is an example of a cellular automaton with a particularly simple `symmetric' transition rule. The sandpile model was first defined in a book for school teachers in 1975 by Engel \cite{Engel}, who called it the \emph{chip firing game}. It was re-discovered in 1989 by the physicists Bak, Tang and Wiesenfeld~\cite{BTW} as a model for gravitation (hence the current name). Later the sandpile model proved to be an extremely useful mathematical object, having connections, besides physics, to group theory, probability theory and partial differential equations. To get a feeling for the sandpile and to enjoy its beauty we refer the reader to the wonderful sandpile applet by Maslov~\cite{Maslov}. The reader is welcome (and urged!) to check the proofs using Maslov's applet. 

The sandpile model can be defined on a general graph by assigning particles to the vertices and toppling a vertex whenever the number of particles assigned to it exceeds the degree of the vertex. In this text, however, we deal only with $\mathbbm{Z}^d$, $d\geq 2$. We shall use $d=2$ for simplicity, however our results generalise easily to $d$ dimensions. We also, purely for aesthetic reasons, imagine the particles living in the cells rather than on the knots. Consequently, we distinguish one cell to be the origin of a cartesian coordinate system and assign coordinates to each cell. 

The word `abelian' stands for the so-called \emph{abelian property} of the sandpile model: if the stabilisation terminates for \emph{some} order of topplings, it terminates for \emph{any} order of topplings. Moreover, for any two orders of topplings that terminate the procedure, the toppled cells are the same up to a permutation and, consequently, the final configurations are identical. This allows us to speak of \emph{the final configuration} resulting from some initial configuration. The proof of this statement is fairly standard and can be found in many references, e.g.~\cite{Meester}.

In this note we consider the following starting configuration: the cell $(0,0)$ has $n$ particles for some $n\geq 4$, every other cell has $2$ particles (we call this \emph{ground level} $2$). It was shown by Fey-den Boer and Redig~\cite{FBR} in 2008 that this model terminates for any $n$. Define the \emph{toppled cluster} to be the set of cells that were toppled at least once during the stabilisation. Similarly, the \emph{visited cluster} are the cells that were either toppled or received a particle during the stabilisation. The toppled and the visited clusters are closely related --- it is easy to see that the latter is obtained by taking the former together with its outer boundary.

An important consequence of the abelian property is the so-called \emph{monotonicity} of the model, which means, adding more particles to the initial configuration can only increase the toppled, and therefore also the visited, cluster. This naturally looking statement needs to be proved rigorously, and this is done by choosing an appropriate order of topplings. `Ignore' the added particles and run the stabilisation as if they were non-existent. We obtain the `old' toppled cluster. Now add the new particles and stabilise the obtained configuration. The toppled cluster can only increase; that proves the monotonicity. 

Define the \emph{radius} of a set of cells $\mathscr{C}$ to be $\sup_{(x,y)\in \mathscr{C}}\max(|x|,|y|)+1$ if $\mathscr{C}$ is non-empty and $0$ if $\mathscr{C}$ is empty. We are interested in the shape and the radius of the toppled cluster for our family of starting configurations. The shape is determined by the following theorem, proved by Fey-den Boer and Redig~\cite{FBR}. 

\begin{theorem}\label{thm: FBR}
For every positive integer $n\ge 4$, the starting configuration with $n$ particles at the origin and ground level $2$ terminates, and the toppled cluster is a square centred at the origin.
\end{theorem}

The result we prove here concerns the radius of $\mathcal{T}_n$. The best previously known bounds of $\sqrt{n}\leq r \leq n/4$ for the radius of the cluster were essentially trivial. 
It was conjectured by Fey-den Boer and Redig~\cite{FBR} in 2008 and suggested by computer simulations that $\sqrt{n}$ is the right order of magnitude. This was proved by by Fey, Levine and Peres~\cite{FLP} in 2010. Our aim here is to give a much simpler proof.

\begin{theorem}\label{thm:sandpilethm}
The toppled cluster of the abelian sandpile for $n$ particles at the origin and ground level $2$ is a square, centred at the origin, of radius at most $4(1+o(1))\sqrt{n}$.
\end{theorem}

The main difficulty in proving Theorem~\ref{thm:sandpilethm} lies in choosing the right approach. Trying to calculate the radius, let alone the number of particles in any given cell {\em exactly}, as an explicit function of $n$, seems to be a very difficult, if not a hopeless task. Even though the final configuration exhibits a certain pattern that recurs for most values of $n$ (see \cite{FLP, FBR}), at the moment of writing we are very far away from understanding it. Therefore, in order to prove Theorem~\ref{thm:sandpilethm} we had to find a method of estimating the radius asymptotically, without referring to the exact numbers of particles in the final configuration. The idea of combining the statements of Lemma~\ref{lem:sand1} and Lemma~\ref{lem:sand2} is harder to discover than it might seem at the first glance.

\begin{proof}
We choose the following order of topplings: first ignore $2$ particles everywhere and run the stabilisation with $n-2$ particles at the origin on ground level $0$. Lemma~\ref{lem:sand1} shows that the toppled cluster of this stabilisation has a radius of order $\sqrt{n}$. After that Lemma~\ref{lem:sand2} makes sure that adding one particle everywhere expands the toppled cluster by at most a factor of $2$. Adding a particle everywhere twice, therefore, can increase the cluster by at most a factor of $4$, which gives the desired bound.

\begin{lemma}\label{lem:sand1}
The radius of the toppled cluster of the abelian sandpile for $n$ particles at the origin and ground level $0$ is at most $(1+o(1))\sqrt{n}$.
\end{lemma}

\begin{proof}
Consider two neighbouring cells of the toppled cluster. Let us call them $a$ and $b$. Without loss of generality we may assume, that the last time $a$ or $b$ was toppled, it was $b$. Then, by the toppling rule, $a$ must contain at least $1$ particle in the final configuration. By a result of Fey-den Boer and Redig~\cite{FBR}, the proof of which we do not present here, the toppled cluster contains a \emph{diamond} of radius $r$, i.e. $D_r=\left\{(x,y): |x|+|y|\leq r-1 \right\}$. Subdivide $D_r$ into $1\times 2$-rectangles, each consisting of two cells; cells that may remain are negligible. We obtain $(1+o(1))r^2$ such $1\times 2$-rectangles. In the final configuration each of them must contain at least one particle. Since the ground level was $0$, all these particles must come from the origin. This shows that $n \geq (1+o(1))r^2$, and thus $r \leq (1+o(1))\sqrt{n}$.
\end{proof}

\begin{lemma}\label{lem:sand2}
The square $S_r$ filled with $4$ particles on ground level $2$ grows by at most a factor of $2$, i.e.\ the toppled cluster of its stabilisation has a radius of at most $2r$.
\end{lemma}

\begin{proof}
More generally, we show that if $S_{r_1}$ is filled with 4's and $S_{r_2}\setminus S_{r_1}$ is filled with 3's on ground level $2$, then the toppled cluster of the stabilisation is contained in $S_{r_1+r_2}$. We do it by induction on $r_1$ for a fixed $r_1+r_2$. If $r_1=0$, then nothing is to prove. Now suppose, $r_1=k$ and we know that the statement is true for smaller values of $r_1$. Topple every cell of $S_{r_2}$ once; this is possible by choosing an appropriate order of topplings, as in the proof of Theorem~\ref{thm: FBR}. Analogously, the resulting configuration consists of $S_k$ full of 4's inside $S_{r_2-1}$ full of 3's, surrounded by a frame of 2's and 1's and further 3's outside. Repeat the procedure of toppling the square filled with 3's and 4's $r_2-k$ times, until we obtain $S_k$ full of 4's, surrounded by a frame of 1's and 2's and stable cells outside. Next, topple every cell in $S_k$ once. Note that so far the toppled cluster is contained in $S_{r_1+r_2}$. The obtained configuration contains $S_{k-1}$ full of 4's, all other cells being stable, and no cell outside $S_{r_2+1}$ has $3$ particles. Therefore, this configuration has fewer particles in each cell than the square of 4's of radius $k-1$ inside the square of 3's of radius $r_2+1$, whose toppled cluster is, by induction hypothesis, contained in $S_{r_1+r_2}$. Hence, by monotonicity, the same holds for the original configuration.
\end{proof}

Now we are ready to prove Theorem~\ref{thm:sandpilethm}. We start out with $n-2$ particles at the origin and $0$ everywhere else, ignoring $2$ particles in each cell; run the sandpile on this configuration. By Lemma~\ref{lem:sand1} we obtain a toppled cluster of radius at most $r_1=(1+o(1))\sqrt{n}$. Now we add one particle everywhere. The obtained configuration has fewer particles in each cell than the square of 4's of size $r_1$ on ground level $2$, therefore, by Lemma~\ref{lem:sand2} and monotonicity the toppled cluster grows by a factor of at most $2$. Now add the second particle everywhere. The obtained configuration has at most as many particles in each cell as the square of 4's of radius $2r_1$ on ground level $2$. Hence, by monotonicity, the toppled cluster grows again by a factor of at most $2$. The radius of the resulting final configuration is thereby at most $4r_1 = 4(1+o(1))\sqrt{n}$. 

\end{proof}

Note that Theorem~\ref{thm:sandpilethm} and its proof can be transferred to $d$ dimensions (for ground level $2d-2$), showing that in $\mathbbm{Z}^d$ the radius of the toppled cluster is at most $c_{d}n^{1/d}$, where $c_d$ is a constant depending solely on $d$.

\section*{Acknowledgements}
\label{sec:acknowledgements}

I would like to thank Béla Bollobás and Rob Morris for drawing this problem to my attention and for their help and advice.

\end{document}